\documentclass[a4paper,12pt,leqno]{amsart} 

\usepackage{amssymb,amsmath} 

\theoremstyle{plain} 
\newtheorem{theorem}{Theorem}[section]
\newtheorem{lemma}[theorem]{Lemma}
\newtheorem{proposition}[theorem]{Proposition}
\newtheorem{corollary}[theorem]{Corollary}

\theoremstyle{definition} 

\newtheorem{remark}[theorem]{Remark}

\newtheorem*{ackn}{Acknowledgment}

\newcommand{\skal}[2]{\langle #1,#2\rangle}

\begin{document}
\baselineskip=17pt

\title{On extremal function of a modulus of a foliation}

\author{M. Ciska}

\address{
Department of Mathematics and Computer Science \endgraf
University of \L\'{o}d\'{z} \endgraf
ul. Banacha 22, 90-238 \L\'{o}d\'{z} \endgraf
Poland
}
\email{mciska@math.uni.lodz.pl}

\date{}

\begin{abstract}
We investigate the properties of a modulus of a foliation on a Riemannian manifold. We give necessary and sufficient conditions for the existence of an extremal function and state some of its properties. We obtain the integral formula which, in a sense, combines the integral over the manifold with integral over the leaves. We state the relation between an extremal function and the geometry of distribution orthogonal to a foliation.  
\end{abstract}

\maketitle

\section{Introduction}

In 1950 Ahlfors and Beurling \cite{ab} introduced a conformal invariant called extremal length of a family of curves in a plane. Its inverse, called the modulus, plays major role in the theory of quasiconformal maps. Modulus was generalized to family of surfaces and submanifolds \cite{bf,kp}. We consider a $p$--modulus of a foliation on a Riemannian manifold. 

The majority of results in this paper is obtained under the assumption of existence of a an extremal function of modulus i.e. a function which realizes the modulus. We give necessary and sufficient conditions for existence of this function (Theorem \ref{extt1}). In particular we consider foliations given by the level sets of submersion (Corollary \ref{extc2}). Moreover, we state some properties of extremal function.

Existence of extremal function allows to define a function 
\begin{equation*}
\hat\varphi(x)=\int_{L_x}f\,d\mu_{L_x},\quad x\in M,
\end{equation*}
for any $\varphi\in L^p(M)$. The main result is the following integral formula (Theorem \ref{intfor})
\begin{equation*}
\int_M f_0^{p-1}\varphi\,d\mu_M=\int_M f_0^p\hat\varphi\,d\mu_M,
\end{equation*}
where $f_0$ is the extremal function for $p$--modulus and $\varphi\in L^p(M)$ is such that ${\rm esssup}|\varphi|<\infty$ and ${\rm esssup}|\hat\varphi|<\infty$. Using this formula we obtain some results concerning the geometry of a foliation. Namely, the tangent gradient of extremal function is related to the mean curvature of distribution orthogonal to $\mathcal{F}$ and there exists a $\mathcal{F}$--harmonic measure on $M$ under the assumption of compactness of a manifold and the leaves.

In the last section we give some examples. We consider a foliation by circles on a torus in $\mathbb{R}^3$, a foliation given by a distance function and a foliation by spheres in a ring in $\mathbb{R}^n$.

\section{Modulus of a foliation}

Let $(M,g)$ be a Riemannian manifold, $\mathcal{F}$ a foliation on $M$. Denote by $\mu_M$ and $\mu_L$ the Lebesgue measure on $M$ and $L\in\mathcal{F}$, respectively. Let $L^p(M)$ denotes the space of all measurable and $p$--integrable functions on $M$ (with the norm $\|\cdot\|_p$). 

\begin{lemma}\label{measurezero}
Let $A\subset M$ be a measurable set. Then, $\mu_M(A)=0$ if and only if 
\begin{equation*}
\mu_M(\{x\in M:\, \mu_{L_x}(A\cap L_x)>0\})=0.
\end{equation*}
In particular, if $f\in L^p(M)$ is nonnegative, then the integral $\int_L f\,d\mu_L$ exists for almost every leaf $L\in\mathcal{F}$.
\end{lemma}
\begin{proof}
Follows by Fubini theorem and the fact that foliation is locally a product $(0,1)^k\times(0,1)^{n-k}$, where $k=\dim \mathcal{F}$, $n=\dim M$.
\end{proof}

In the space $L^p(M)$ consider a family ${\rm adm}_p(\mathcal{F})$ of all nonnegative functions $f$ such that $\int_L f\,d\mu_L\geq 1$ for almost every $L\in\mathcal{F}$. Functions belonging to this family are called {\it admissible}. The $p$--{\it modulus} of $\mathcal{F}$ is defined as follows
\begin{equation*}
{\rm mod}_p(\mathcal{F})=\inf_{f\in{\rm adm}_p(\mathcal{F})}\|f\|_p
\end{equation*}
if ${\rm adm}_p(\mathcal{F})\neq\emptyset$ and ${\rm mod}_p(\mathcal{F})=\infty$ otherwise.

We say that $f_0\in{\rm adm}_p(\mathcal{F})$ is {\it extremal} for $p$--modulus of a foliation $\mathcal{F}$ if ${\rm mod}_p(\mathcal{F})=\|f_0\|_p$. 

\begin{proposition}\cite{bf,bl1}\label{propmod}
The modulus has the following properties.
\begin{enumerate}
\item If $\mathcal{L}\subset\mathcal{F}$ and $\bigcup\mathcal{L}$ is measurable, then ${\rm mod}_p(\mathcal{L})\leq{\rm mod}_p(\mathcal{F})$.
\item If $\mathcal{F}=\bigcup_i\mathcal{L}_i$ and $\bigcup\mathcal{L}_i$ is measurable for all $i$, then ${\rm mod}_p(\mathcal{F})^p\leq\sum_i{\rm mod}_p(\mathcal{L}_i)^p$.
\item ${\rm mod}_p(\mathcal{F})=0$ if and only if there is $f\in{\rm adm}_p(\mathcal{F})$ such that $\int_L f\,d\mu_L=\infty$ for almost every $L\in\mathcal{F}$.
\item If $f\in L^p(M)$, then there is a subfamily $\mathcal{L}$ such that ${\rm mod}_p(\mathcal{L})=0$ and $f\in L^1(L)$ for almost every $L\in\mathcal{F}\setminus\mathcal{L}$.
\item If $f_n\to f$ in $L^p(M)$, then there is a subsequence $(f_{n_i})$ of $(f_n)$ and a subfamily $\mathcal{L}\subset\mathcal{F}$ such that ${\rm mod}_p(\mathcal{L})=0$ and $f_{n_i}\to f$ in $L^1(L)$ for almost every $L\in\mathcal{F}\setminus\mathcal{L}$. 
\end{enumerate}
\end{proposition}

\section{Existence and properties of extremal function}

Let $M$ be a Riemannian manifold, $\mathcal{F}$ a foliation on $M$.

\begin{theorem}\label{extt1}
There is an extremal function $f_0$ for $p$--modulus of $\mathcal{F}$ if and only if for any subfamily $\mathcal{L}\subset\mathcal{F}$ such that $\mu(\bigcup\mathcal{L})>0$ we have ${\rm mod}_p(\mathcal{L})>0$.
\end{theorem}
\begin{proof} $(\Rightarrow)$ Suppose there is a family $\mathcal{L}\subset\mathcal{F}$ such that $\mu(\bigcup\mathcal{L})>0$ and ${\rm mod}_p(\mathcal{L})=0$. By Proposition \ref{propmod} there is an admissible function $f$ for $\mathcal{L}$ such that $\int_L f=\infty$ for almost every leaf $L\in\mathcal{L}$. Moreover, there is $n>0$ such that $\|f_0\|_p>\|\frac{1}{n}f\|_p$ on $\bigcup\mathcal{L}$. Put
\[
\tilde{f}_0=\left\{\begin{array}{cc} f_0 & \textrm{on $M\setminus\bigcup\mathcal{L}$} \\ \frac{1}{n}f & \textrm{on $\bigcup\mathcal{L}$} \end{array}\right..
\]
Then $\tilde{f}_0$ is admissible for $\mathcal{F}$ but $\|f_0\|_p>\|\tilde{f}_0\|_p$. Contradiction.

$(\Leftarrow)$  It suffices to show that ${\rm adm}_p(\mathcal{F})$ is closed in $L^p(M)$. Take the sequence $(f_n)$ of admissible functions convergent to $f$. Then $f\in L^p(M)$ and $f\geq 0$. By Proposition \ref{propmod} there is a subsequence $(f_{i_j})$ and a subfamily $\mathcal{L}$ such that ${\rm mod}_p(\mathcal{E})=0$  and $\int_L f_{n_i}d\mu_L\to\int_L f d\mu_L$ for every $L\in\mathcal{F}\setminus\mathcal{L}$. Hence $\int_L f d\mu_L\geq 1$ for every $L\in\mathcal{F}\setminus\mathcal{L}$. Since by assumption $\mu(\bigcup\mathcal{L})=0$, we get that $\int_L f d\mu_L\geq 1$ for almost every leaf $L\in\mathcal{F}$. Thus $f$ is admissible.
\end{proof}

Put
\begin{align*}
F_{\infty}=\{L\in\mathcal{F}:\, \mu_{L}(L)=\infty\}\quad\textrm{and}\quad M_{\infty}=\bigcup\mathcal{F}_{\infty}
\end{align*}

\begin{corollary}\label{extc1}
If $\mu_M(M)<\infty$ and there is an extremal function for $p$--modulus of $\mathcal{F}$ for some $p>1$, then $\mu(M_{\infty})=0$.
\end{corollary}
\begin{proof} Function $f=1$ satisfies the conditions of Proposition \ref{propmod}(3) for $\mathcal{F}_{\infty}$. Hence ${\rm mod}_p(\mathcal{F}_{\infty})=0$. Thus, by Theorem \ref{extt1}, $\mu(M_{\infty})=0$.
\end{proof}

We will specify the conditions for the existence of an extremal function in the case of a foliation given by a level sets of a submersion. 

Let $\Phi:M\to N$ be a submersion between Riemannian manifolds. Then there is a decomposition
\begin{equation*}
TM=\mathcal{V}^{\Phi}\oplus\mathcal{V}^{\Phi},
\end{equation*}
where $\mathcal{V}^{\Phi}=\ker\Phi_{\ast}$ and $\mathcal{H}^{\Phi}=(\mathcal{V}^{\Phi})^{\bot}$. The differential $\Phi_{\ast}:\mathcal{H}^{\Phi}\to TN$ is a linear isomorphism. Denoting its dual as $\Phi_{\ast}^{\top}:TN\to\mathcal{H}^{\Phi}$ we define the Jacobian $J\Phi$ of $\Phi$ as follows
\begin{equation*}
J\Phi=\sqrt{\det(\Phi_{\ast}^{\top}\circ\Phi_{\ast}:\mathcal{H}^{\Phi}\to\mathcal{H}^{\Phi})}.
\end{equation*}
We will need the following version of Fubini theorem \cite{ic,sw}
\begin{equation}\label{fubini}
\int_M f\,d\mu_M=\int_{y\in N}\left( \int_{\Phi^{-1}(y)}\frac{f}{J\Phi}\,d\mu_{\Phi^{-1}(y)}\right)\,d\mu_N,
\end{equation}  
for any nonnegative and measurable function $f$.

\begin{proposition} \label{extp0}
Assume a foliation $\mathcal{F}$ is given by a submersion $\Phi:M\to N$ such that $J\Phi<C$ for some $C$. Let $\mathcal{L}\subset\mathcal{F}$. If ${\rm mod}_p(\mathcal{L})=0$ for some $p>1$, then $\mathcal{L}\subset\mathcal{F}_{\infty}$.
\end{proposition}
\begin{proof}
Suppose $\mathcal{L}\not\subset\mathcal{F}_{\infty}$. Then there is a subfamily $\mathcal{L}'$ of $\mathcal{L}$ such that $\mu(\bigcup\mathcal{L}')>0$ and $\mu_L(L)<\infty$ for $L\in\mathcal{L}'$. By Proposition \ref{propmod}(3) there is an admisible function $f$ such that $\int_L f=\infty$ for almost every leaf $L\in\mathcal{L}$. Thus by H\"older inequality we have
\begin{equation*}
\int_L f d\mu_L =\int_L\frac{f}{(J\Phi)^{\frac{1}{p}}}(J\Phi)^{\frac{1}{p}} d\mu_L\leq\bigg( \int_L\frac{f^p}{J\Phi} d\mu_L \bigg)^{\frac{1}{p}} \bigg( \int_L(J\Phi)^{\frac{q}{p}} d\mu_L \bigg)^{\frac{1}{q}},
\end{equation*}
where $\frac{1}{p}+\frac{1}{q}=1$. Hence
\begin{equation} \label{extp0e}
\int_L\frac{f^p}{J\Phi} d\mu_L=\infty\quad\textrm{for almost every $L\in\mathcal{L}'$}.
\end{equation}
By \eqref{extp0e} and Fubini theorem \eqref{fubini} 
\[
\|f\|_p^p=\int_Mf^p d\mu_M=\int_N\big( \int_{\Phi^{-1}(y)}\frac{f^p}{J\Phi}d\mu_{\Phi^{-1}(y)}\big) d\mu_N=\infty.
\]
Contradiction ends the proof.
\end{proof}
\begin{corollary} \label{extc2}
Assume a foliation $\mathcal{F}$ is given by a submersion $\Phi:M\to N$. Assume $J\Phi<C$ for some $C$ and $\mu_M(M)<\infty$. Then there is an extremal function for $p$--modulus of $\mathcal{F}$ (for any $p>1$) if and only if $\mu(M_{\infty})=0$.
\end{corollary}
\begin{proof}($\Rightarrow$) Follows by Corollary \ref{extc1}.

($\Leftarrow$) Assume $\mu(M_{\infty})$=0. Take $\mathcal{L}\subset\mathcal{F}$ such that ${\rm mod}_p(\mathcal{L})=0$. Then $\mathcal{L}\subset\mathcal{F}_{\infty}$  by Proposition \ref{extp0}. Hence $\mu(\bigcup\mathcal{L})\leq \mu(M_{\infty})=0$. Therefore Theorem \ref{extt1} implies that there is an extremal function for $\mathcal{F}$. 
\end{proof}

In the end of this section we state some properties of an extremal function.

\begin{proposition} \label{pef}
Let $f_0$ be an extremal function for $p$--modulus of a foliation $\mathcal{F}$. Then
\begin{enumerate}
\item $\int_L f_0 d\mu_L=1$ for almost every leaf $L\in\mathcal{F}$,
\item $f_0$ is positive.
\end{enumerate}
\end{proposition}
\begin{proof}
$(1)$ Put $M_{f_0}=\{x\in M:\, \int_{L_x}f_0\,d\mu_{L_x}<\infty\}$. By Proposition \ref{propmod}(3) and Theorem \ref{extt1} $\mu_M(M\setminus M_{f_0})=0$. Hence we may assume $M=M_{f_0}$. Consider a function
\begin{equation*}
f(x)=\frac{f_0(x)}{\int_{L_x} f_0\,d\mu_{L_x}},\quad x\in M.
\end{equation*}
Then $f$ is admissible, $f\leq f_0$ and $\|f\|_p\leq \|f_0\|_p$. Since $f_0$ is extremal, we have $\|f\|_p=\|f_0\|_p$. Therefore $f=f_0$, so $\int_L f_0d\mu_L=1$ for almost every leaf $L\in\mathcal{F}$.

$(2)$ Suppose there is a set $A$ of positive and finite measure such that $f_0=0$ on $A$. Put 
\begin{align*}
A_1 &=\{x\in A :\, 0<\mu_{L_x}(L_x\cap A)<\infty\}, \\
A_2 &=\{x\in A :\, \mu_{L_x}(L_x\cap A)=\infty\}, \\
A_3 &=\{x\in A :\, \mu_{L_x}(L_x\cap A)=0\}.
\end{align*}
By Lemma \ref{measurezero} $\mu_M(A_3)=0$. Hence we may assume $A=A_1\cup A_2$. Let $0\leq t\leq 1$ and put
\begin{equation*}
f(x)=\left\{\begin{array}{ll} \frac{1-t}{\mu_{L_x}(L_x\cap A_1)} & x\in A_1 \\ tf_0(x) & x\in M\setminus A_1 \end{array}\right.
\end{equation*}
Then $f\geq 0$ and $\int_L f d\mu_L=1$ for almost every leaf $L\in\mathcal{F}$. Moreover,
\begin{equation*}
\|f\|_p^p=(1-t)^pC+t^p\|f_0\|_p^p,\quad\textrm{where $C=\int_{A_1}\frac{1}{\mu_{L_x}(L_x\cap A_1)^p}d\mu$}.
\end{equation*}
Hence $\|f\|_p<\|f_0\|_p$ if and only if the function
\[
\alpha(t)=(1-t)^pC+t^p\|f_0\|_p^p-\|f_0\|_p^p,\quad 0\leq t\leq 1,
\]
is negative at some $t$. Existence of such $t$ follows from the fact that $\alpha(1)=0$ and $\alpha '(1)>0$.
\end{proof}

\begin{remark}
Proposition \ref{pef}(1) was first established in \cite{bl2} but obtained under the assumption of continuity of extremal function.
\end{remark}

\section{The integral formula}

Let $M$ be a Riemannian manifold, $\mathcal{F}$ a foliation on $M$. Assume there exists an extremal function $f_0$ for $p$--modulus of $\mathcal{F}$. 

\begin{lemma}\label{pto1}
If $\varphi\in L^p(M)$, then $\varphi\in L^1(L)$ for almost every leaf $L\in\mathcal{F}$.
\end{lemma}
\begin{proof}
Follows immediately by Theorem \ref{extc1} and Proposition \ref{propmod}(4). 
\end{proof}

By above lemma, for any $\varphi\in L^p(M)$ the following function 
\begin{equation*}
\hat\varphi(x)=\int_{L_x}\varphi\,d\mu_{L_x},\quad x\in M
\end{equation*}
is well defined.

\begin{theorem}\label{intfor}
Assume ${\rm essinf}_Mf_0>0$. Then there is the following integral formula
\begin{equation} \label{pefe1}
\int_M f_0^{p-1}\varphi \,d\mu_M=\int_M f_0^p\hat{\varphi} \,d\mu_M,
\end{equation}
for any function $\varphi\in L^P(M)$ such that ${\rm esssup}|\varphi|<\infty$ and ${\rm esssup}_M|\hat{\varphi}|<\infty$.
\end{theorem}

\begin{proof}
Let $\varphi\in L^p(M)$, ${\rm esssup}|\varphi|<\infty$ and assume $|\hat\varphi|<C$ for some $C>0$. Put $t_0=\frac{1}{2C}$. For any $t\in[0,t_0)$ consider functions
\begin{equation*}
f^{\pm}_t(x)=\frac{f_0(x)\pm t\varphi(x)}{\widehat{f_0\pm t\varphi}(x)}=\frac{f_0(x)\pm t\varphi(x)}{1\pm t\hat\varphi(x)},\quad x\in M.
\end{equation*}
Then $f^{\pm}_t\in L^p(M)$ and for sufficiently small $t$, $f^{\pm}_t\geq 0$. Clearly $\int_L f^{\pm}_t\,d\mu_L\geq 1$ for almost every $L\in\mathcal{F}$. Hence $f^{\pm}_t\in{\rm adm}_p(\mathcal{F})$. Therefore
\begin{equation}\label{extfune0}
\|f_0\|_p\leq \|f^{\pm}_t\|_p.
\end{equation}
Fix $x\in M$ and define a function $\delta(t)=f^{\pm}_t(x)^p$. Then $\delta(0)=f_0(x)^p$. Since $\delta$ is smooth, by mean value theorem, there is $\theta(t)\in(0,t)$ such that $\delta(t)-\delta(0)=\delta '(\theta(t))t$. Moreover
\begin{equation}\label{extfund}
\delta '(t)=\pm p\frac{(f_0(x)\pm t\varphi(x))^{p-1}(\varphi(x)-f_0(x)\hat\varphi(x))}{(1\pm t\hat\varphi(x))^{p+1}}
\end{equation}
Hence
\begin{equation*}
f^{\pm}_t(x)^p-f_0(x)^p=\pm p\frac{(f_0(x)\pm \theta(t)\varphi(x))^{p-1}(\varphi(x)-f_0(x)\hat\varphi(x))}{(1\pm \theta(t)\hat\varphi(x))^{p+1}}t.
\end{equation*}
We have $1-t_0|\hat\varphi(x)|>1-t_0C=\frac{1}{2}$. Therefore $\frac{1}{1-t_0|\hat\varphi(x)|}<2$. Thus
\begin{align*}
\left|\frac{f^{\pm}_t(x)^p-f^{\pm}_t(x)^p}{t} \right| &\leq p|\varphi(x)-f_0(x)\hat\varphi(x)|\frac{(f_0(x)+t_0|\varphi(x)|)^{p-1}}{(1-t_0|\hat\varphi|)^{p+1}}\\
&\leq 2^{p+1}p(|\varphi(x)|+Cf_0(x))(f_0(x)+t_0|\varphi(x)|)^{p-1}.
\end{align*}
The function $\alpha(x)=(|\varphi(x)|+Cf_0(x))(f_0(x)+t_0|\varphi(x)|)^{p-1}$ is integrable on $M$. Indeed, by H\"older inequality,
\begin{align*}
\int_M g d\mu_M &\leq 
2^{p+1}\left(p\int_M(|\varphi| +f_0 C)^pd\mu_M\right)^\frac{1}{p}
 \left(\int_M((f_0+ t_0|\varphi|)^{p-1})^qd\mu_M\right)^\frac{1}{q}\\
&= 2^{p+1}\left(p\int_M(|\varphi| +f_0 C)^pd\mu_M\right)^ \frac{1}{p}\left(\int_M(f_0+ t_0|\varphi|)^{p}d\mu_M\right)^ \frac{1}{q}< \infty,
\end{align*}
where $\frac{1}{p}+\frac{1}{q}=1$. Hence, by Lebesgue dominated convergence theorem, \eqref{extfund} and \eqref{extfune0}, we have
\begin{align*}
0 &\leq \lim_{t\to 0^+}\int_M \frac{f_t^p-f_0^p}{t}\, d\mu_M\\
&=\int_M\lim_{t\to 0^+}\frac{f_t^p-f_0^p}{t}\, d\mu_M\\
&=\pm p\int_M f_0^{p-1}(\varphi-f_0\hat{\varphi})\, d\mu_M.
\end{align*}
Hence \eqref{pefe1} holds.
\end{proof}

As a corollary we obtain a formula, firstly proven in \cite{kp} using different approach, for $p$--modulus and an extremal function $f_0$ of a foliation given by the level sets of a submersion. 

\begin{corollary} \label{pef5}
Assume there is an extremal function $f_0$ for $p$--modulus of $\mathcal{F}$ and that ${\rm essinf}_Mf_0>0$. If $\mathcal{F}$ is given by a submersion $\Phi:M\to N$, then
\begin{equation} \label{pefe4}
f_0=\frac{J\Phi^{\frac{1}{p-1}}}{\widehat{J\Phi^{\frac{1}{p-1}}}} \quad\textrm{and}\quad {\rm mod}_p(\mathcal{F})= \left(\int_N \left(\widehat{ J \Phi ^{\frac{1}{p-1}}}\right)^{1-p}\,d\mu_N\right)^{\frac{1}{p}}. 
\end{equation}
\end{corollary}
\begin{proof}
By Theorem \ref{intfor} and Fubini theorem \eqref{fubini}, for any $p$--integrable function $\varphi$ such that ${\rm esssup}_M|\varphi|<\infty$ and ${\rm esssup}_M|\hat{\varphi}|<\infty$ we have
\begin{gather*}
\begin{split}
\int_M f_0^{p-1}\varphi \,d\mu_M &=\int_M \left( f_0^p\int_{\Phi^{-1}(y)}\varphi \,d\mu_{\Phi^{-1}(y)} \right)\, d\mu_M \\
&=\int_{y\in N} \int_{\Phi^{-1}(y)}\left( \frac{f_0^p}{J\Phi}\int_{\Phi^{-1}(y)}\varphi\, d\mu_{\Phi^{-1}(y)} \right)\, d\mu_N \\
&=\int_{y\in N}\left( \int_{\Phi^{-1}(y)}\varphi\, d\mu_{\Phi^{-1}(y)}\right)\left( \int_{\Phi^{-1}(y)}\frac{f_0^p}{J\Phi}\,d\mu_{\Phi^{-1}(y)} \right)\, d\mu_N \\
&=\int_{y\in N}\int_{\Phi^{-1}(y)} \left( \varphi\int_{\Phi^{-1}(y)}\frac{f_0^p}{J\Phi}\,d\mu_{\Phi^{-1}(y)} \right)\, d\mu_N \\
&=\int_M \varphi J\Phi\left( \int_{\Phi^{-1}(y)}\frac{f_0^p}{J\Phi}\,d\mu_{\Phi^{-1}(y)} \right)\, d\mu_M.
\end{split}
\end{gather*}
Hence
\[
\int_M f\varphi\, d\mu_M=0,\quad f=f_0^{p-1}-J\Phi\left( \int_{\Phi^{-1}(y)}\frac{f_0^p}{J\Phi}\,d\mu_{\Phi^{-1}(y)} \right).
\]
We will show that $f=0$. Suppose $f>0$ on a set $A$ of positive and finite measure. Clearly $f<\infty$ on $A$. Let $\varphi$ be a function on $M$ equal to $f_0$ on $A$ and zero elsewhere. Then $\varphi\in L^p(M)$, $\varphi\geq 0$, ${\rm esssup}_M|\varphi|<\infty$ and ${\rm esssup}_M|\hat{\varphi}|<\infty$. Since $f_0>0$, we get
\begin{equation*}
0=\int_M f\varphi\, d\mu_M=\int_A ff_0\, d\mu_M>0.
\end{equation*}
Contradiction. Therefore $f=0$ and
\begin{equation} \label{pefpe1}
\frac{f_0^{p-1}}{J\Phi}=\int_{\Phi^{-1}(y)}\frac{f_0^p}{J\Phi}\,d\mu_{\Phi^{-1}(y)}.
\end{equation}
Denote the right hand side of above equality by $\alpha$. Then $\alpha$ is a function on $M$ constant on leaves. Moreover,
\begin{equation} \label{pefpe2}
f_0=\alpha^{\frac{1}{p-1}}(J\Phi)^{\frac{1}{p-1}}.
\end{equation}
Hence, by \eqref{pefpe1}
\[
\alpha=\alpha^{\frac{p}{p-1}}\int_{\Phi^{-1}(y)}(J\Phi)^{\frac{1}{p-1}} d\mu_{\Phi^{-1}(y)}.
\]
Therefore
\[
\alpha^{-\frac{1}{p-1}}=\int_{\Phi^{-1}(y)} (J\Phi)^{\frac{1}{p-1}} d\mu_{\Phi^{-1}(y)},
\]
so by \eqref{pefpe2} the firs condition of \eqref{pefe4} holds.  Again, by Fubini theorem \eqref{fubini} 
\begin{align*}
{\rm mod}_p(\mathcal{F})^p=&\|f_0\|_p\\
=&\int_M \left( \frac{J\Phi^\frac{1}{p-1}}{\widehat{J\Phi^\frac{1}{p-1}}}\right)^pd\mu_M\\
=&\int_N\Bigg(\frac{\int_{\Phi^{-1}(y)}J\Phi^{\frac{p}{p-1}-1}d\mu_{\Phi^{-1}(y)}}{\widehat{\left(J\Phi^\frac{1}{p-1}d\mu_L\right)^p}}\Bigg)d\mu_N\\
=&\int_N\left(\widehat{J\Phi^{\frac{1}{p-1}}} \right)^{1-p}d\mu_N.
\end{align*}
\end{proof}

By integral formula \eqref{intfor} we obtain some results concerning the geometry of a foliation $\mathcal{F}$.

\begin{corollary} \label{pef3}
Let $\mathcal{F}$ be a foliation on $M$ with closed leaves. Assume that extremal function $f_0$ for $p$--modulus of $\mathcal{F}$ exists and is $C^2$--smooth and ${\rm essinf}_Mf_0>0$. Then the mean curvature $H_{\mathcal{F}^{\bot}}$ of distribution $\mathcal{F}^{\bot}$ orthogonal to $\mathcal{F}$ is of the form
\begin{equation} \label{pefe2}
H_{\mathcal{F}^{\bot}}=(p-1)(\nabla(\ln f_0))^{\top}.
\end{equation}
In particular, $f_0$ is constant on leaves of $\mathcal{F}$ (hence $f_0=\frac{1}{\hat 1}$) if and only if distribution $\mathcal{F}^{\bot}$ is minimal.
\end{corollary}
\begin{proof}
Let $X$ be any compactly supported vector field tangent to $\mathcal{F}$. Denote by ${\rm div}_{\mathcal{F}}X$ the divergence of $X$ with respect to the leaves of $\mathcal{F}$. Then
\[
{\rm div}_{\mathcal{F}}X={\rm div}X+\skal{H_{\mathcal{F}^{\bot}}}{X}.
\]
Since for any smooth function $f$ we have
\[
{\rm div}(fX)=f{\rm div}X+\skal{\nabla f}{X},
\]
then putting $\varphi={\rm div}_{\mathcal{F}}X$ in \eqref{pefe1} we get
\begin{gather*}
\begin{split}
0 &=\int_M \big( f_0^{p-1}({\rm div}X)+f_0^{p-1}\skal{H_{\mathcal{F}^{\bot}}}{X} \big) d\mu\\
&=\int_M \big( {\rm div}(f_0^{p-1}X)-\skal{\nabla f_0^{p-1}}{X}+f_0^{p-1}\skal{H_{\mathcal{F}^{\bot}}}{X} \big) d\mu \\
&=\int_M \big( f_0^{p-1}\skal{H_{\mathcal{F}^{\bot}}}{X}-(p-1)f_0^{p-2}\skal{\nabla f_0}{X} \big) d\mu \\
&=\int_M f_0^{p-2}\skal{f_0 H_{\mathcal{F}^{\bot}}-(p-1)\nabla f_0}{X} d\mu.
\end{split}
\end{gather*}
Taking $X=\psi(f_0 H_{\mathcal{F}^{\bot}}-(p-1)\nabla f_0)^{\top}$, where $\psi$ is a function on $M$ with compact support, we get \eqref{pefe2}. Moreover, by \eqref{pefe2} $\mathcal{F}^{\bot}$ is minimal if and only if $f_0$ is constant along leaves, hence $f_0=\frac{1}{\hat 1}$.
\end{proof}

\begin{corollary}\label{pef4}
Let $\mathcal{F}$ be a Riemannian foliation on $M$ with closed leaves. Assume that extremal function $f_0$ for $p$--modulus of $\mathcal{F}$ exists and is $C^2$--smooth and ${\rm essinf}_Mf_0>0$. Then $f_0$ is constant on leaves of $\mathcal{F}$.
\end{corollary}

\begin{proof}
Follows from the Corollary \ref{pef3} and the fact that the distribution orthogonal to Riemannian foliation is totally geodesic, hence minimal.
\end{proof}

It appears that extremal function $f_0$ defines a harmonic measure. Precisely, we say that a measure $\mu$ on a Riemannian manifold $M$ with a foliation $\mathcal{F}$ is $\mathcal{F}$--{\it harmonic} if
\begin{equation*}
\int_M \Delta_{\mathcal{F}}f\,d\mu=0
\end{equation*}
for all smooth functions $f$ on $M$, where $\Delta_{\mathcal{F}}f={\rm div}_{\mathcal{F}}(\nabla f)^{\top}$.

\begin{corollary}\label{pef6}
Let $M$ be a closed manifold, $\mathcal{F}$ a foliation on $M$ with closed leaves. Assume there exists an extremal function $f_0$ for $p$--modulus of $\mathcal{F}$ such that ${\rm essinf}_M f_0>0$. Then the measure 
\begin{equation*}
\mu=f_0^{p-1}\mu_M
\end{equation*}
is $\mathcal{F}$--harmonic.
\end{corollary}

\begin{proof}
Let $f$ be a smooth function on $M$. Putting $\varphi=\Delta_{\mathcal{F}}f$ in integral formula \eqref{intfor}, since $\hat\varphi=0$, we have
\begin{equation*}
0=\int_M f_0^{p-1}\Delta_{\mathcal{F}}f\,d\mu=\int_M \Delta_{\mathcal{F}}f\,d\mu.
\end{equation*}
\end{proof}

\section{Examples}

\subsection{Foliation by spheres}

Let $M=\{x\in\mathbb{R}^n:\, r_1<|x|<r_2\}$ be a ring. Consider a foliation $\mathcal{F}$ by spheres of radii $r\in(r_1,r_2)$. Then $\mathcal{F}$ is given, in a spherical coordinates, by a submersion $\Phi(r,\alpha_1,\ldots,\alpha_{n-1})=r$. Thus, an extremal function $f_0$ for a $p$--modulus of $\mathcal{F}$ is, by Corollary \ref{pef5}, equal to
\begin{equation*}
f_0(r,\alpha_1,\ldots,\alpha_{n-1})=Cr^{1-n},
\end{equation*} 
for some positive constant $C$. Hence $f_0$ is constant on the leaves of $\mathcal{F}$. Moreover, $\mathcal{F}$ is Riemannian.

\subsection{Foliation by circles on a torus in $\mathbb{R}^3$}

Consider a torus $\mathbb{T}$ in $\mathbb{R}^3$ given parametrically
\[
(\alpha,\beta)\mapsto ((R+r\cos\alpha)\cos\beta,(R+r\cos\alpha)\sin\beta,r\sin\alpha)
\]
and let $\mathcal{F}$ be a foliation on $\mathbb{T}$ given by a submersion $\Phi(\alpha,\beta)=\beta$. Then $\mathcal{F}$ is a foliation by circles of radius $r$. After some computations we get that extremal function $f_0$ for $p$--modulus of $\mathcal{F}$ is of the form
\[
f_0=C(R+r\cos\alpha)^{\frac{-1}{p-1}},
\] 
for some positive constant $C$. Hence, by Corollary \ref{pef3},
\[
H_{\mathcal{F}^{\bot}}=\frac{\sin\alpha}{R+r\cos\alpha}\frac{1}{r}\frac{\partial}{\partial\alpha},
\]
which is the curvature of orthogonal foliation $\mathcal{F}^{\bot}$ of circles of radii $R+r\cos\alpha$ in torus $\mathbb{T}$.

\subsection{Foliation by a distance function}

Let $L_0$ be a closed hypersurface in $\mathbb{R}^n$. Let $U$ be any of two connected complements of $\mathbb{R}^n\setminus L_0$. Consider a distance function from $L_0$,
\[
\rho(x)={\rm dist}(x,L_x),\quad x\in\mathbb{R}^n.
\]
Let $\mathcal{F}$ be a family of level sets of $\rho$. There is a neighborhood $M$ of $L_0$ in $U$, in which $\rho$ is smooth and $\mathcal{F}$ is a foliation on $M$, see \cite{rf}. Moreover $|\nabla\rho|=1$ on $M$ \cite{dz}. By Corollary \ref{pef5} an extremal function for $p$--modulus of $\mathcal{F}$ is of the form
\[
f_0(x)=\frac{1}{\hat 1(x)},
\]  
hence is constant on the leaves. Clearly $\mathcal{F}$ is Riemannian and the orthogonal distribution $\mathcal{F}^{\bot}$ is totally geodesic (compare Corollary \ref{pef4}).

\begin{ackn} This article is based on a part of author's PhD Thesis. The author wishes to thank her advisor Professor Antoni Pierzchalski who suggested the problem under consideration in this paper and for helpful discussions. In addition, the author wishes to thank Kamil Niedzia\l omski for helpful discussions that led to improvements of some of the theorems. 
\end{ackn}

\end{document}